\documentclass[draft]{amsart}
\everymath{\displaystyle}  

\usepackage{amsmath}
\usepackage{amssymb}
\usepackage{amscd}
\usepackage{pb-diagram}

\usepackage[all]{xy}

\newcounter{rem}
\newenvironment{remark}{\addtocounter{rem}{1}
\begin{enumerate} \item[] {\bf Remark~\therem . }}
{\end{enumerate}\bigskip}

\newtheorem{theorem-introduction}{Theorem}
\newtheorem{corollary-introduction}{Corollary}
\newtheorem{definition}[subsubsection]{Definition}
\newtheorem{lemma}[subsubsection]{Lemma}
\newtheorem{corollary}[subsubsection]{Corollary}
\newtheorem{theorem}[subsubsection]{Theorem}
\newtheorem{proposition}[subsubsection]{Proposition}
\theoremstyle{definition}

\def\supp{\ensuremath{\mathrm supp}}

\def\Ker{\ensuremath{\mathrm Ker}}

\def\PP{{\mathbb  P}}
\def\NN{{\mathbb  N}}
\def\ZZ{{\mathbb  Z}}
\def\CC{{\mathbb  C}}
\def\RR{{\mathbb  R}}

\def\HH{{\mathbb  H}}

\def\VN{ \text{M}}

\def\Ran{{\mathrm Ran}}

\def\U{\ensuremath{{\mathcal{U}}}}

\def\H{\ensuremath{{\mathcal{H}}}}

\def\cC{\ensuremath{{\mathcal{C}}}}
\def\ncC{\ensuremath{{n\mathcal{C}}}}
\def\cD{\ensuremath{{\mathcal{D}}}}

\newcommand{\Adh}[1]{\ensuremath{\overline{#1}}}

\newcommand{\dlog}[2]{\ensuremath{\Omega^{#1}_{X}(log #2)}}

\def\dbar{\ensuremath{{\overline{\partial}}}}

\title[]{A factorisation theorem for curves with vanishing self-intersection}
\author{P. Dingoyan}
\address{P. Dingoyan\\
 Universit\'e Paris 6 \\case 247, 4 place Jussieu, 75252 Paris Cedex 05, France.}
 \email{pascal.dingoyan@imj-prg.fr}

\begin{document}

\bibliographystyle{plain}


\maketitle

\section{introduction}
\subsection{ }
We give, in some cases, a positive answer to a question of F. Campana:
\begin{trivlist}
\item Let $i:C_{0}\to Y$ be a smooth curve in a compact K\"{a}hler surface. 
Assume:\begin{trivlist}
\item[i)] the self-intersection of $C_{0}$ is vanishing $C_{0}.C_{0}=0$, 
\item[ii)] the image of the fundamental group of $C_{0}$ in $Y$ is of infinite index $[\pi_{1}(Y):i_{*}\pi_{1}(C_{0})]=+\infty$.
\end{trivlist}
\item
\item
{\bf Question:} Is  $C_{0}$ a fiber of a holomorphic map $f:Y\to B$ from $Y$ to a curve ?
\end{trivlist}
\bigskip

Let $p: X\to Y$ be the infinite covering of $Y$ with fundamental group $i_{*}\pi_{1}(C_{0})$. Then $i:C_{0}\to Y$ lifts to an embedding $s:C_{0}\to X$ with  $s(C_{0}).s(C_{0})=0$.
We give the following positive answer:
\begin{theorem} Assume that a finite covering of $p:X\to Y$ is Galois or  
admits a proper Green function. Then there exists proper holomorphic maps $f':X\to B'$ and $f: Y\to B$ to curves $B$ and $B'$ such that $s(C_{0})$ is a fiber of $f'$ and $C_{0}$ is a fiber of $f$. In particular $X$ is holomorphically convex.
 \end{theorem}
The theorem applies in particular when $i_{*}\pi_{1}(C_{0})$ is finite and $\pi_{1}(Y)$ is infinite.
We give below a more general statement for an effective divisor in a K\"{a}hler compact manifold.
\subsection{ }
The strategy is to solve the Poincar\'{e}-Lelong equation $i\partial\dbar u=s(C_{0})$ (in the sense of currents) and then use the logarithmic form $\partial u$ to produce the required fibration. In general, the obstruction to solve this equation lies in the first Chern class of the line bundle $[s(C_{0})]$. Our basic observation is that $s(C_{0}).s(C_{0})=0$ implies that the first Chern class of $[s(C_{0})]$ vanishes in cohomology. This fact is a consequence of the Hodge index theorem in complete K\"{a}hler manifolds of infinite volume (see sect. \ref{section Hodge index}). In particular: 
\begin{proposition} Let $C_{0}$ be a divisor in a smooth K\"{a}hler surface which satisfies $i)$ and $ii)$ above. Then its first Chern class $c_{1}([C_{0}])\in H^{2}(Y,\ZZ)$ lies in $H^{2}(\pi_{1}(Y),\ZZ)$.
\end{proposition}

We then find conditions which imply that the $\partial\dbar-$lemma of K\"{a}hler geometry is valid in the non compact setting: in the case of a Galois cover, we use the Galois $\partial\dbar-$lemma that was proved by the author in \cite{Din2013}. In the case where a proper Green function exists, we follow the approach of Gromov \cite{GroKahHyp} and Napier-Ramachandran \cite{NapRam1995}, \cite{NapRam2001}:
 to solve the $\partial\dbar-$equation, it is enough to solve the Laplace equation on
  functions and integrate by parts if the geometry allows.
  
  \subsection{Comparison with some previous works }
In \cite{GroKahHyp}, \cite{NapRam1995}, \cite{NapRam2001}, the authors deduce factorisation theorems from the study of an analytic $1-$form associated to a square integrable or  compactly supported $1-$form. 
 Here one uses a $2-$form with compact support and a positivity assumption and produces a logarithmic $1-$form. Then the study of the foliation follows almost the same pattern than the aformentioned articles. However, in the case of a Galois covering, the use of the automorphism group allows a more elementary construction.
 
M. Nori \cite{Nor} studied the image of the fundamental group of the normalization $\pi_{1}(\bar C)\to \pi_{1}(Y)$ of a non smooth curve $C$ whose irreducible components have positive self-intersection. This lead to a weak Lefschetz theorem: the image is of finite index.  From this point of view, the weak lefschetz theorem was reconsidered in \cite{NapRam1998}.

F. Campana \cite{CamNeg} studied consequences of the negativity of the self-intersection of a compact curve in an infinite covering and gave some factorisation theorems (see loc. cit. sect. 4): if a smooth curve in a K\"{a}hler compact surface satisfies $ii)$ and has a torsion normal bundle then $C$ is a fiber of a holomorphic map. His proof used Ueda's theory \cite{Ued} describing a neighborhood of a smooth compact curve with vanishing self-intersection in a complex surface. 

B.Tortaro \cite{Tor2000} and J.Pereira \cite{Per2006} proves some factorisation theorem for divisors with some vanishing hypothesis on Chern class. For instance,  assume  three smooth disjoint divisors in a manifold $X$ lie in the same rational cohomology class. Then they prove there exists a holomorphic map from $X$ to a curve such that each divisor is a fiber.
B.Tortaro \cite{Tor2000} gives examples of two smooth disjoint divisors in a surface $X$ which lie in the same rational cohomology class but which are not fiber of a map to a curve. One can check that the image of the fundamental group of each divisor is of finite indice in the fundamental group of $X$.

The topic of Shafarevitch conjecture and Shafarevich dimension is discussed in \cite{Cam1994}, \cite{Kol1995}. A different point of view is given by
Gurjar and his co-authors (see \cite{GurKol}, \cite{GurPur}): they studied some fibrations $p:Y\to B$ of compact  surfaces and proved that the image of the fundamental group of any fiber in the fundamental group of $Y$ is finite.
\subsection{Acknowledgement.}
I warmly thank F.Campana for the discussions around this subject. He is at the root of the question and discussions with him lead to some of the conclusions. Our contribution is mainly technical.
 I thank also the complex analysis team of Nancy for its pleasant welcome to its seminar.

\section{preliminaries}
In this section we recall that a local proper fibration always extends to a global proper fibration in a manifold of bounded geometry. We refer to  \cite{GroKahHyp} or \cite{ShubinNantes} for a definition of a manifold of bounded geometry. We note that a covering of a compact Riemannian manifold with the pull-back structure is of bounded geometry.
\subsection{Notations} If $D$ is a divisor, let $|D|$ be its support and let $[D]$ be the line bundle it defines.
We first recall: 
\begin{lemma}\label{proper}  Let $f:X\to Y$ be a continuous map between locally compact spaces. Let $p\in X$ be such that  $f^{-1}f(p)$ is compact. Then there exists neighborhoods $U$ of $f^{-1}f(p)$ and $U_{1}$ of $f(p)$ such that $f:U\to U_{1}$ is proper.
\end{lemma}
\begin{proof} see \cite{Stein}  p.77 or S\"{a}tz 9 or \cite{BarHulPet} p.27. 
\end{proof}
\begin{definition} Let $D_{1}$ and $D_{2}$ be two divisors in a complex manifold $X$. Assume $D_{1}$ is compact. If $\dim_{\CC}X\geq 3$, let $\omega$ be a fixed K\"{a}hler form. Let $h$ be a smooth metric on $[D_{1}]$ such that the Chern form $c_{1}(h)$ has compact support. Then we defined the intersection number by
 \begin{eqnarray}\label{self-intersection}
 D_{1}.D_{2}:=\int_{X}c_{1}(D_{1})\wedge c_{1}(D_{2})\wedge \omega^{n-2}\,.
 \end{eqnarray}
 \end{definition} 
 From \cite{Dem1993MonAmp} sect.2 (in particular 2.11 and 2.12), one can check that the intersection number of divisors without common irreducible components is positive.
In particular if $\{D_{i}\}$ is a (finite) family of compact effective divisor such that $\cup_{i} |D_{i}|$ is connected then the 
  generalized Zariski's lemma holds for the intersection form on $\{D_{i}\}$ (see \cite{CamNeg} appendix A.A or \cite{BarHulPet} Chap.1 sect.2). 
  \begin{lemma}[Zariski's lemma]\label{Zariski's lemma}
  Let $(X,Y)\mapsto X.Y$ be a bilinear form defined on $\oplus_{i=1}^{n}\RR e_{i}$ such that $e_{i}.e_{j}\geq 0$ if $i\not=j$ and the matrix $(e_{i}.e_{j})_{i,j}$ is irreducible (the graph, with an edge between each pair $\{e_{i},\,e_{j}\}$ such that $e_{i}.e_{j}\not=0$, is connected). Then only one of the following possibility occurs for the quadratic form it defines:
  \item[i)] $Q$ is negative definite.
  \item[ii)] $Q$ is semi-definite of rank $n-1$, with a one dimensional annihilator spanned by a strictly positive vector $X_{0}>0$ (i.e. its coordinate are positive). 
  \item[iii)] $Q(X)>0$ for some  $X$, then there exists $X_{0}>0$ such that $\forall i,\, X_{0}.e_{i}>0$.
  \end{lemma}
%
One deduces:

\begin{lemma}\label{fibre}  Let $D$ be an effective divisor with connected compact support $|D|$ in a complex manifold $X$. If $\dim_{\CC}X\geq 3$, assume there exists a K\"{a}hler form $\omega$ and define self-intersection of a compact divisor as in (\ref{self-intersection}).
Assume $D.D=0$ and there exists an holomorphic map $f: X\to \Delta$ to the unit disc. Then $|D|$ is a connected component of a fiber of $f$. Hence $f$ is proper in a neighborhood of $|D|$. 
Moreover if $D$ is not a multiple of an effective divisor and $f(|D|)=0$, there exists $l\in \NN^{*}$ such that the divisor defined by the fiber over $0$ is equal to $l.D$

\end{lemma}
%
The following theorem is noteworthy: a local proper holomorphic map  to a curve extends to a global proper holomorphic map. The second point says that a divisor moves if it moves in a covering. We refer to \cite{GroKahHyp}, \cite{AraBreRam}, \cite{NapRam1995}, \cite{NapRam1998} and \cite{Tor2000} for variations on this theorem. 
For the sake of completness, we sketch a proof using Barlet's space. 
\begin{theorem}  \label{deformation} \item[1)] Let $H\to X$ be a compact hypersurface in a K\"{a}hler complex manifold  of bounded geometry.  Assume that there exists a neighborhood $W$ of $H$ and an holomorphic map $f:W\to \Delta$ to the unit disc such that $f^{- 1}(0)=H$. Then there exists a curve $B$ and a proper map $a:X\to B$ extending $f$. In particular $X$ is holomorphically convex.
\item[2)] Let $D_{0}$ be a divisor in a compact K\"{a}hler manifold $Y$. Assume $D_{0}.D_{0}=0$ and $D_{0}$ is not a sum of non trivial effective divisors with this property. Assume there exists a covering $p: X\to Y$ such that a compact connected component 
$T$ of $p^{-1}(|D_{0}|)$ admits a non constant holomorphic function on one of its neighborhood. Then there exists a holomorphic map $b:Y\to B$ to a curve such that a multiple of $D_{0}$ is a fiber.
\end{theorem}
\begin{proof}\item[1.0)] Using lemma \ref{proper}, we may assume that $f$ is proper. Let $f^{*}(0)$ be the divisor defined by the fiber over $0$. 
Let $D$ be the smallest effective divisor supported on $f^{-1}(0)$ such that $D.D=0$. Zariski's Lemma \ref{Zariski's lemma} implies $f^{*}(0)=l.D$ for some $l\in \NN^{*}$. Let $\cC_{l.D}(X)$ be the connected component of Barlet's space \cite{Bar} of $X$ which contained $l.D$. Let $n:\ncC_{l.D}(X)\to \cC_{l.D}(X)$ be a normalization of this space. If $s$ is a point of Barlet's space, let $D_{s}$ be the associated cycle.
From the definition of Barlet's space, the map $f$ defines a morphism  $i:\Delta\to \cC_{l.D}(X)$ that maps $0$ to $l.D$. Shrinking $\Delta$ if necessary, we may assume that a fiber over $\Delta^{*}$ is smooth. 
\item[1.1)] From \cite{Bar} Lemma 1 p.37 and the fact that a cycle $D_{s}$ is homologuous to $l.D$, we deduce that
   $i:\Delta\to \cC_{l.D}(X)$ is a local parametrization and $\cC_{l.D}(X)$ is a smooth curve at $l.D$.

\item[1.2)] 
Let $\cD=\{(x,s)\in  X\times \ncC_{D}(X) \text{ such that } x\in D_{s}\}$ be the universal divisor with associated projections $f_{1}$, $f_{2}$.
From \cite{Bar} th1 p.38 and $1.1)$, we deduce that $D_{s}$ is reduced and irreducible if $s$ is a generic point in $ \ncC_{D}(X)$. But if $s,s'\in \ncC_{D}(X)$, then $D_{s}.D_{s'}=0$, 
hence for generic $s,s'$, these divisors are either disjoint or equal. This implies:
\begin{eqnarray*}
ii)\,\cD\to X \text{ is finite}& \text{and} &ii)\, \cD\to X \text{ is generically one to one.}
\end{eqnarray*}

 \item[1.3)] The main point is that $\cD\to X$ is proper, hence onto, if $(X,\omega)$ is of bounded geometry: in such a manifold, homologuous irreducible cycles have bounded volume and bounded diameter.  Hence Bishop's compactness theorem may be applied. For a proof, see e.g. \cite{Cam1994} prop.3.12. 

\item[1.4)]
 The map $\cD\to X$ is one to one and a biholomorphism for $X$ is normal. Let $g=f_{1}^{-1}:X\to \cD$.
 Then $a=f_{2}\circ g:X\to \ncC_{D}(X)$ is a proper holomorphic map such that $f_{2}\circ g_{|W}=i\circ f$ and $f^{-1}(0)=l.D$.

\item[2)] Let $p:X\to Y$ be an etale covering. Assume a connected component $T$ of $p^{-1}(|D_{0}|)$ is compact. Let $D$ be the divisor with support $|D|=T$ and structure scheme defined by $p^{*}(D_{0})$.

Let $W'$ be a neighborhood of $|D|$ such that $p'=p_{|W'}:W'\to p(W')$ is a finite covering of order $m$ and $W'$ is the connected component of $p^{-1}p(W')$ which contains $|D|$ (one uses a retract of a neighborhood of $|D_{0}|$ onto $|D_{0}|$).

 Let $\sigma_{i}(f)$ be the $i-$th symetric function of $f$ with respect to $p'$. Then $\sigma_{i}(f)$ vanishes on $|D_{0}|=p(|D|)$. 
But $D_{0}.D_{0}=0$ implies that 
$(\sigma_{i}(f)=0)=\alpha_{i}D_{0}+L_{i}$, where $L_{i}$ is a divisor with support disjoint from $|D_{0}|$.  Using that $D$ is a connected component of a pullback, one easily sees that $\alpha_{m}=m.l$.

We deduce from the first case of the theorem that $\cC_{m.l.D_{0}}(Y)$ is a smooth curve at $m.l.D_{0}$ and there exists
   a map $b:Y\to \ncC_{m.l.D_{0}}(Y)$  induced by $\sigma_{m}(f)$. 
\end{proof}
\section{The Hodge index theorem for square integrable $(1,1)-$forms.}\label{section Hodge index}
\subsection{The Hodge index theorem}

This section contains our main lemma.
\begin{lemma}\label{Hodge index} Let $(X,\omega)$ be a complete K\"{a}hler manifold of infinite volume. 
\item[1)] Any $l^{2}-$harmonic $(1,1)-$form is primitive. In particular the intersection form is negative definite on $\H^{1,1}_{(2)}(X)$ the space of square integrable harmonic $(1,1)-$forms.
\item[2)] Let $\alpha$ be a closed $(1,1)-$form with compact support such that $\int_{X}\alpha\wedge \bar \alpha\wedge \omega^{n-2}\geq 0$, then $\alpha$ is orthogonal to the harmonic forms and $\alpha$ is $d-$exact. 
\end{lemma}

\begin{proof} see also \cite{Din2014} for $(1)$ and $(2)$.
For notational convenience, if $a,b$ are $(1,1)-$forms on $X$, one sets $\int a\wedge b:=\int_{X}a\wedge b\wedge \omega^{n-2}$.
\item[1)] Let $h\in \H^{1,1}_{(2)}(X)$ then $\Lambda h\in \H^{0}_{(2)}(X)$. However square integrable harmonic functions are identically nul on a complete manifold of infinite volume. This is a consequence of Gaffney theorem (\cite{Gaf}, \cite{AndVes} p.93, \cite{Rha} th.26). Hence $\Lambda h$ is vanishing. This prove that the form $h\wedge\bar h\wedge \omega^{n-2}$ is negative and $||h||^{2}=-h\wedge\bar h\wedge \omega^{n-2}$ (see \cite{Huy} sect. 1.2).
\item[2)] Let $\alpha=h+e$ be the decomposition of $\alpha$ in its harmonic part $h$ plus a form $e\in\Adh{\Ran(d)}={\Ker d^{*}}^{\perp}$.
Then $\int \alpha\wedge\bar\alpha=\int h\wedge \bar h+\int e\wedge \bar e+2\text{Re}\int e\wedge \bar h$.
But $\int e\wedge \bar h=\pm(e,\star h)=0$ for $\star h$ is harmonic if $h$ is. 
Also $\int e\wedge \bar e=\int \alpha\wedge \bar e=\lim_{i}\int \alpha\wedge \Adh{d\theta_{i}}$ for the form is compactly supported. Integrating by parts, one sees that $\int e\wedge \bar e=0$. Hence $-||h||^{2}=\int h\wedge \bar h\geq 0$.

 To conclude, we use that if a closed form with compact support is orthogonal to the $l^{2}-$harmonic forms then it is exact (see e.g. \cite{Luc} lemma 1.92).
%
\end{proof}
\subsection{Negativity of self-intersection of compact divisors}
As a corollary, we obtain a general result on self-intersection of a compact divisor in a complete K\"{a}hler manifold of infinite volume (see \cite{NapRam1998}, \cite{CamNeg} sect. 2).
\begin{corollary}\label{negativity}\item[1)] Let $X$ be a complete K\"{a}hler manifold of infinite volume. Let $C$ be a compact divisor. Then $C.C\leq 0$ and  if $C.C=0$ then the cohomology class $c_{1}([C])$ is vanishing in $H^{2}(X,\RR)$.
\item
\item[2)]Let $Y$ be a compact K\"{a}hler manifold. Let $i:C\to Y$ be a divisor in $Y$ such that $i_{*}\pi_{1}( C )$ is of infinite index in  $\pi_{1}(Y)$ and $C.C=0$. Then $c_{1}([C])\in H^{2}(Y,\ZZ)$ belongs to $H^{2}(\pi_{1}( Y ),\ZZ )$. Hence the pullback of this class to the universal covering is trivial.

\end{corollary}

\begin{proof}\item[1)] is a direct consequence of the lemma.
\item[2)]
 Let $\tilde C_{i}$ be a connected component of $\pi^{-1}( C )$. Let $f_{i}:X_{i}=  X^{u}/G_{i}\to Y$ be the quotient of the universal covering $X^{u}$ of $Y$ by the stabiliser of $\tilde C_{i}$. Then $ f_{i}: \tilde C/G_{i}=C_{i}\to C$ is a compact divisor in $X_{i}$ which is biholomorphic to $C$, so that $C_{i}.C_{i}=0$. This implies that $c_{1}([C_{i}])$ is trivial in $H^{2}(X_{i},\RR)$ hence also $c_{1}([\tilde C_{i}])=f_{i}^{-1}c_{1}([C_{i}])$. Therefore
 the cohomology class of $\pi^{*}(C)$ is trivial.
The  class $c_{1}([C])$ defines a class in $H^{2}(\Pi_{1}(Y),\ZZ)$ (cohomology classes on $Y$ which are exact on $X^{u}$). 
\end{proof}
We may compare this with the case of a fibration: let $f:Y\to W$ be a holomorphic fibration onto a curve of genus greater than one. Then $W$ is a $K(G,1)$, hence $\ZZ=H^{2}(W, \ZZ)\simeq H^{2}(\pi_{1}(W),\ZZ)$. We obtain a natural morphism $H^{2}(W, \ZZ)\to H^{2}(\Pi_{1}(Y),\ZZ)$. It maps the cohomology of a point $p\in W$, to the cohomology of the fiber $f^{-1}( p )$ in $Y$. 
\begin{corollary}[Weak Lefschetz]  Let $Y$ be a compact complex K\"{a}hler manifold. Let $i:C\to Y$ be a divisor in $Y$ such that $C.C>0$. Then the index of $i_{*}\pi_{1}(|C|)$ in $\pi_{1}(Y)$ is finite.
\end{corollary}
\section{The case of a compact divisor in an infinite Galois covering.}
\subsection{Notations and definitions.}
Let $p:X\to Y$ be a covering. In general, the space of square integrable sections of a pull-back bundle will be defined through the pull-back metrics. 
\begin{trivlist}
\item[a)]
Let $\Gamma$ be the deck transformation group of the covering $p:X\to Y$. Then $\CC[\Gamma]$ acts on $l^{2}(X)$ by a representation $\lambda$ such that $\lambda(\sum_{g\in \Gamma}a_{g}\delta_{g}).(F)(x)=\sum_{g\in \Gamma}a_{g}F(g^{-1}x)$. 
The weak closure (or the bicommutant) of $\lambda(\CC[\Gamma])$ is given by the action of the Von Neumann algebra $\VN(\Gamma)$ of the discret group $\Gamma$. 

\item[b)] An element $\lambda_{s}(f)\in \VN(\Gamma)$ is caracterised by its value $f=\lambda_{s}(f)(\delta_{e})\in l^{2}(\Gamma)$ on the Dirac mass at the neutral element $\delta_{e}$. Then $\lambda_{s}(f): l^{2}(\Gamma)\to l^{2}(\Gamma)$ is given by the left convolution of $f$ on $l^{2}(\Gamma)$ and $\lambda(f): l^{2}(X)\to l^{2}(X)$ is defined through the weak limit $\sum_{g\in \Gamma}f(g)\lambda(\delta_{g}).F$ in $l^{2}(X)$.
\item[c)]By definition, an element of $\VN(\Gamma)$ is a weak isomorphism (or is almost invertible) if it is injective with dense range. The fact that $\VN(\Gamma)$ admits a finite trace implies that it is enough to be injective or with dense range.

\end{trivlist}
We will use the Von Neumann algebra $\VN(\Gamma)$ through the following lemma (\cite{Din2013} cor.3.4.6): 
\begin{lemma}[A Galois $\partial\dbar-$lemma]\label{A ddbar-lemma} Let $p:X\to Y$ be a Galois covering of a compact K\"{a}hler manifold. Let $\alpha$ be a $d-$closed square integrable $(p,q)-$form on $X$ which is orthogonal to harmonic forms. Then there exists a weak isomorphism $r\in \VN(\Gamma)$, there exists a square integrable form $\beta$ on $X$ such that $\lambda(r).\alpha=\partial\dbar \beta$. 
\end{lemma}

\subsection{Factorisation in the Galois case}

\begin{theorem}\label{galois} Let $p: X\to Y$ be an infinite Galois covering of a compact K\"{a}hler manifold. Let $ C$ be a compact divisor in $ X$ such that $C.C=\int_{X}c_{1}([C])^{2}\wedge\omega^{n-2}=0$. Then
\item[1)] there exists a proper map $f': X\to B'$ to a curve such that $C$ is a fiber (up to multiplicity). In particular $X$ is holomorphically convex. 
\item[2)] Moreover there exists a holomorphic map $f:Y\to B$ to a curve such that  $p_{*}(C)$ is a  fiber of $f$ (up to multiplicity).
\end{theorem}

\begin{proof} From \ref{negativity}, we may assume $C$ connected and effective.
\item[a)]We first prove that $1)$ implies $2)$. Using theorem \ref{deformation} (2), it is enough to prove that  $|C|$ is a connected componant of $p^{-1}p(|C|)$:  if not, another irreducible component $T$ in the connected component of $p^{-1}p(C)$ which contains $C$ is compact for $p$ is Galois. If $T$ intersects $C$ then $T.C>0$. However the Hodge index theorem  and $C.C=0$ implies that $c_{1}(C)=du$ is exact. But $c_{1}(T)$ is closed with compact support,
hence $$T.C=\int_{X}c_{1}(T)\wedge du\wedge\omega^{n-2}=0.$$ 
\item[b)]
We may assume that $C$ is not a multiple of an effective divisor, and that $C$ is not the sum of non trivial compact effective divisors with vanishing self-intersection. Let $W$ be a small connected neighborhood of $p(|C|)$ such that $W'$, the connected component of $p^{-1}(W)$ which contains $ C$, satisfies $p^{-1}p(C)\cap W'=C$. Then $p'=p_{|W'}:W'\to W$ is a finite covering. Hence the direct image $p_{*}(C)$ is a divisor with vanishing self-intersection. Zariski's lemma \ref{Zariski's lemma} implies there exists a minimal effective divisor $C_{0}$ with support $p(|C|)$ such that $C_{0}.C_{0}=0$. Then $p'^{*}(C_{0})=C$.

Let $s$ be a section of $[C]$ such that $(s)=C$.
We fix a metric on the line bundle $[C]$ such that $|s|=1$ outside a given compact neighborhood of $C$. Using a pull-back metric by $p'$,
we may even assume that $|s|$ is invariant by the stabiliser subgroup of $C$. 

Let $\alpha$ be the metric first Chern form of $[C]$. The Poincar\'{e}-Lelong equation reads $$dd^{c}\ln|s|^{2}=C-\alpha\,.$$

Lemma \ref{Hodge index} proves that $\alpha$ is orthogonal to the $l^{2}-$harmonic forms for $0=C.C=\int_{ X}\alpha\wedge\alpha$.
Let $\Gamma$ be the deck transformations group of $p:X\to Y$. The Galois $\partial\dbar-$lemma \ref{A ddbar-lemma} implies that there exists a weak isomorphism $\lambda(f)\in \VN(\Gamma)$ and a function $u$ square integrable in the domain of $d$ and $dd^{c}$, such that: $$\lambda(f).\alpha=dd^{c}u\,.$$
   
 The function $\ln |s|^{2}$ has compact support, hence $\lambda(f).\ln|s|^{2}$ is well defined. We obtain the following equation (the action of the Von Neumann algebra on a current with compact support is defined by duality):
\begin{eqnarray} 
\lambda(f).C&=&\lambda(f).\left( dd^{c}\ln|s|^{2}+\alpha\right) \\
                       &=&dd^{c} \left( \lambda(f).\ln|s|^{2}+u \right)
\end{eqnarray}
This proves that the function $P= \left( \lambda(f).\ln|s|^{2}+u \right)$ is pluriharmonic on $ X\setminus \Gamma. |C|$ with a logarithmic pole along $\Gamma. | C |$  ($P$ extends has  a difference of two plurisubharmonic functions for $f$ is not positive in general).

Note that $d^{c}P$ is square integrable in the fiber of $j^{*}p=p_{| X\setminus\Gamma| C|}$ and is closed over $Y\setminus |C_{0}|$. It defines a non trivial class of weight two in $\HH^{1}(Y\setminus| C_{0}|,p_{*(2)}\CC)$  for its residue along $p^{-1}( |C_{0}| )$ is $\lambda(f).[ C]$, which is a non vanishing current (for $\lambda(f)$ is injective). The degeneracy of $\U(\Gamma)-$spectral sequence implies that the $(1,0)-$part of this class is represented by a square integrable logarithmic one form $L$. From \cite{Din2013},\cite{Din2014}, we know that $L$ is closed. 

In our situation, the form $\partial P=\partial  \left( \lambda(f).\ln|s|^{2}+u \right)$ is a representative of the $(1,0)-$part of $L$: it is a logarithmic one form which is closed on $ X\setminus p^{-1}(|C_{0}|)$ and square integrable as a section of $p^{*}\dlog{1}{|C_{0}|}$.

Using translation, one sees that the space $A$ of closed square integrable logarithmic one forms is infinite dimensional for it is non trivial. Hence, for any finite number of compact connected component $C_{1},\ldots,\, C_{k}$ of $p^{-1}( |C_{0}| )$, the linear map $\oplus_{i=1}^{k} Res_{C_{i}}:A\to \oplus_{i=1}^{k}\CC$ has a non trivial kernel.
 A logarithmic meromorphic form with vanishing residue on $C_{1},\ldots,\, C_{k}$ is holomorphic. This form is therefore a closed holomorphic one form on $( X\setminus p^{-1}( |C_{0}| ))\cup   C_{1}\ldots\cup C_{k}$.

\begin{lemma} Let $X'$ be a complex manifold. Assume that a space $\Lambda$ of closed holomorphic one forms is infinite dimensional. Then for any compact subset $K$, there exists a non constant holomorphic function in a neighborhood of $K$ whose derivative lies in $\Lambda$.
\end{lemma}
\begin{proof}  $X'$ is exhausted by relatively compact domains $U$ with smooth boundaries. Hence $H_{1}(\bar U)$ is finite dimensional and the natural map given by integration $\Lambda\to H_{1}(\bar U)^{*}$ has a non trivial kernel. A closed holomorphic one form with  vanishing period is exact, on $\bar U$. Its primitive is $\dbar-$closed.
\end{proof}

We conclude the proof of the theorem using lemma \ref{fibre} and theorem \ref{deformation} $(1)$ and $(2)$.

\end{proof}

\begin{corollary} Let $i:C\to Y$ be a divisor in a compact K\"{a}hler manifold. Assume that a subgroup of finite index of $i_{*}\pi_{1}(|C|)$ has a normal closure of infinite index in $\pi_{1}(Y)$ and that $C.C=0$. Then each connected component of the support of $C$ is a fiber of a holomorphic map to a curve.
\end{corollary}
\begin{proof}Hypothesis implies that there exists an infinite Galois covering $p:X\to Y$ such that each connected componant of $p^{-1}(|C|)$ is compact. We conclude using the previous theorem and the Hodge index theorem \ref{negativity}.
\end{proof}

\begin{corollary}[Campana] Let $i:C\to Y$ be a curve with vanishing self-intersection in a surface and assume that $i_{*}\pi_{1}(C)$ is finite. Then the universal covering of $X$ is holomorphically convex.
\end{corollary}
\section{Manifolds of bounded geometry with a proper Green function.}


In this section we give a proof using potential theory, following the theory developped in  \cite{Ram1996}, \cite{NapRam1995}, \cite{NapRam2001}. we will however restrict ourselves to a non compact covering $(X,\omega)$ of a compact K\"{a}hler manifold. 
\subsection{Solving the $\partial\dbar-$equation from the $\Delta-$equation.}
\begin{lemma}
 Let $( X,\omega)$ be a complete K\"{a}hler manifold  which admits a Green function. Let $\alpha$ be a closed $(1,1)-$form with compact support such that $\int_{X}\alpha\wedge \bar\alpha\wedge\omega^{n-2}=0$ then $\alpha$ is $\partial\dbar-$exact. Moreover if the Green function vanishes at infinity, a solution of $\partial\dbar u=\alpha$ exists such that $\lim_{x\to \infty}u(x)=0$.
\end{lemma}
\begin{remark}

 In case $\lambda_{1}$, the first eigenvalue of $\Delta$ on square integrable functions, is strictly positive then the function $u$ will belong to the domain of $d$ acting on square integrable functions.
\end{remark}

\begin{proof}
We may assume that $\alpha$ is real. Let $G$ be the Green function of the hyperbolic complete K\"{a}hler manifold $(X,\omega)$. Then $G(\Lambda \alpha)=u$ is a function such that $\Lambda \alpha=-\Delta u=\Lambda i\partial\dbar u$. Moreover $du$ is square integrable (\cite{NapRam1995} lemma 1.6 p.815).
Hence the closed form  $\alpha-dd^{c} u$ is pointwise primitive. Let $\theta$ be a positive smooth function with compact support such that $\theta=1$ on a neighborhood of $supp(\alpha)$. Then:
\begin{eqnarray}
0\geq \int_{X}\theta\left(\alpha-dd^{c}u\right)\wedge \Adh{\left(\alpha-dd^{c}u\right)} &= & -\int_{X}dd^{c}\theta(du\wedge d^{c}u). 
\end{eqnarray}
There exists a positive exhaustion function $f$ on $X$ with bounded gradient and bounded hessian for $(X,\omega)$ has bounded geometry. Let $\theta =\chi(\epsilon f)$ where $\chi$ is a smooth positive function on $\RR$, $\chi_{|]-\infty,1]}=1$, $\chi_{|[2,+\infty[}=0$. Letting $\epsilon$ goes to zero, one sees that $ \int_{X}\left(\alpha-dd^{c}u\right)\wedge \Adh{\left(\alpha-dd^{c}u\right)}=0$.
Therefore
$$\alpha=dd^{c}u\,.$$ 
\end{proof}

\subsection{Setting} \bigskip
We consider a compact divisor $ C$ in a complex K\"{a}hler manifold of bounded geometry admitting a Green function. We assume that $C.C=0$.

With the notations of the lemma, we know that $C=dd^{c}\ln|s|^{2}+\alpha=dd^{c }(\ln|s|^{2}+u)$ is cohomologically trivial. Assume moreover the Green function is proper then
\begin{trivlist}
\item[i)] the plurisubharmonic function $(\ln|s|^{2}+u)$ is strictly negative on $X$,
\item[ii)] the set $\{\ln|s|^{2}+u\leq -a\}$ $(a>0)$ are compact subsets in $X$.
\end{trivlist}
Indeed, $u$ is pluriharmonic on $X\setminus \supp(\alpha)$ 
 and it vanishes at infinity (if $\lambda_{1}>0$ then $u_{|X\setminus \supp(\alpha)}$ is a square integrable pluriharmonic function in a manifold of bounded geometry, hence it vanishes at infinity). Therefore $(\ln|s|^{2}+u)$ is strictly negative for it is non constant. The second assertion follows for $\supp(\ln|s|^{2})$ is compact. 
 
 In the case of a proper fibration to a curve $f:X\to B$, the function $(\ln|s|^{2}+u)$ is the pull-back of the Green function with a pole on $f(C)$.
 
 \begin{remark} \label{remarque}
 Any non compact irreducible analytic set of strictly positive dimension in $X$ intersects $\supp(\alpha)$:
assume the opposite, $u$ restricted to such a set would be pluriharmonic and would vanish at infinity. Hence $u$ would be vanishing and $ \ln|s|^{2}+u$ would be constant. In particular $ C$ intersects any non compact positive dimensional irreducible analytic set in $X$ for we can shrink $\supp(\alpha)$ in an arbitrary neighborhood of $ C$.
 \end{remark}

We will use the following lemma (see e.g. \cite{NapRam2001} lemma 2.1 p.389 or \cite{Per2006} for a statement using foliations):
\begin{lemma} Let $u, v$ be pluriharmonic functions defined on a complex manifold $M$. Assume the level set of $v$ are compact or empty. Then $\partial u\wedge \partial v=0$ and
\item[1)] either there is a non closed leaf and $\partial u=c\partial v$ for some constant $c$,
\item[2)] or every leaf is closed, there exists a holomorphic map $f:M\to \PP^{1}$ such that $f$ restricted to the leaf is constant and $\partial u=f\partial v$.
\end{lemma}
We deduce a second version of a factorisation theorem for divisors with vanishing self-intersection.
\begin{proposition} Let $C_{0}\to Y$ be an effective divisor with vanishing self-intersection in a compact K\"{a}hler manifold $(Y,\omega_{0})$. Assume that $G=\mathrm{Im}(\pi_{1}(C_{0})\to \pi_{1}(Y))$ has infinite index . Let $p: X\to Y$ be the covering of $Y$ with fundamental group $G$.

Assume that $(X,p^{*}\omega_{0})$ admits a proper Green function (e.g. $\lambda_{1}$, the first eigenvalue of the laplacian acting on functions, is strictly positive). Then there exists a holomorphic map $f:Y\to B$ such that a multiple of $C_{0}$ is a multiple of a fiber of $f$.
\end{proposition}

\begin{proof} We may assume that $C_{0}$ is connected and minimal with the above properties. By hypothesis there exists a divisor $C$ in $ X$ that is mapped biholomorphically to $C_{0}$ by $p$.

Then there exists a neighborhood $W'$ of $ C$ which is biholomorphic to a neighborhood $W$ of $C$. The existence of a section $\sigma:W\to W'$, defines a plurisubharmonic function  $v=(\ln|s|^{2}+u)\circ \sigma $ on $W$ which is pluriharmonic away from $C_{0}$.
Hence $C_{0}$ admits a basis of Levi flat neighborhoods.  Let $( C_{j})_{j\in \NN^{*}}$ be the connected components of $p^{-1}(C_{0})$. Assume that $ C= C_{1}$. Remark \ref{remarque} implies that $ C_{k}$ is compact. Note that $p$ is an infinite covering hence there are an infinite number of connected component.

For $a>A$ large enough, a connected component $B_{k}(-a)$ of $p^{-1}(\{v<-a\})$ contains only one connected component $ C_{k}$ of $p^{-1}(C_{0})$.
Let $k\not=1$. Let $v_{k}=v\circ p_{|B_{k}(-A)}$ be the restriction to a neighborhood of $ C_{k}$ of  the pullback  by $p$ of $v$. 

If the leaves of $\{v_{k}=-a\}$ were non integrable, there would exist a complex number $c$ such that on $B_{k}(-A)\setminus C_{k}$, $\partial u=c \partial v_{k}$. However $\partial v_{k}$ has a logarithmic pole on $ C_{k}$ but $u$ is smooth on a neighborhood of $C_{k}$. Hence the leaves are closed and there exists a holomorphic map $f_{k}: B_{k}(-A)\setminus C_{k}\to \PP^{1}$ such that $\partial u=f_{k}\partial v_{k}$. But $u$ is smooth and $\partial v_{k}$ has a logarithmic pole on $C_{k}$, hence $f_{k}$ extends as a holomorphic map to $B_{k}(-A)$ vanishing on $C_{k}$. We apply theorem \ref{deformation} to conclude that a multiple of $C_{k}$ and of $C_{0}$ are fibers of non constant holomorphic maps to a curve.

\end{proof}

The corollary applies in particular if  $i_{*}(\pi_{1}(C))$ is amenable and $\pi_{1}(Y)$ is not amenable using \cite{Bro}. 
Also it applies in the case of a Galois covering with Deck transformation group $G$ which is not a finite extension of $\ZZ$ or $\ZZ^{2}$ (see   the arguments in \cite{Ram1996}, \cite{NapRam2001} p.399).

\begin{remark}
F. Campana indicated to us that the original problem has a positive answer in the case of a smooth rational curve or a smooth elliptic curve in a complex K\"{a}hler surface.  The proof is algebraic and uses the Kodaira's classification.
\end{remark}
%

\end{document}